\documentclass[11pt]{article}
\usepackage{epsf}
\usepackage{amsmath}
\usepackage{epsfig}
\usepackage{times}
\usepackage{amssymb}
\usepackage{amsthm}
\usepackage{setspace}
\usepackage{cite}

\usepackage{algorithmic}  
\usepackage{algorithm}

\usepackage{shadow}
\usepackage{fancybox}
\usepackage{fancyhdr}

\def\x{{\bf x}}

\def\x{{\mathbf x}}

\def\x{{\bf x}}

\def\h{{\bf h}}

\def\be{\begin{equation}}
\def\ee{\end{equation}}
\def\ba{\left[\begin{array}}
\def\ea{\end{array}\right]}

\def\x{{\bf x}}

\def\1{{\bf 1}}

\def\g{{\bf g}}
\def\0{{\bf 0}}

\newtheorem{theorem}{Theorem}

\newtheorem{lemma}{Lemma}

\begin{document}

\begin{singlespace}

\title {Negative spherical perceptron 
}
\author{
\textsc{Mihailo Stojnic}
\\
\\
{School of Industrial Engineering}\\
{Purdue University, West Lafayette, IN 47907} \\
{e-mail: {\tt mstojnic@purdue.edu}} }
\date{}
\maketitle

\centerline{{\bf Abstract}} \vspace*{0.1in}

In this paper we consider the classical spherical perceptron problem. This problem and its variants have been studied in a great detail in a broad literature ranging from statistical physics and neural networks to computer science and pure geometry. Among the most well known results are those created using the machinery of statistical physics in \cite{Gar88}. They typically relate to various features ranging from the storage capacity to typical overlap of the optimal configurations and the number of incorrectly stored patterns. In \cite{SchTir02,SchTir03,TalBook} many of the predictions of the statistical mechanics were rigorously shown to be correct. In our own work \cite{StojnicGardGen13} we then presented an alternative way that can be used to study the spherical perceptrons as well. Among other things we reaffirmed many of the results obtained in \cite{SchTir02,SchTir03,TalBook} and thereby confirmed many of the predictions established by the statistical mechanics. Those mostly relate to spherical perceptrons with positive thresholds (which we will typically refer to as the positive spherical perceptrons). In this paper we go a step further and attack the negative counterpart, i.e. the perceptron with negative thresholds. We present a mechanism that can be used to analyze many features of such a model. As a concrete example, we specialize our results for a particular feature, namely the storage capacity. The results we obtain for the storage capacity seem to indicate that the negative case could be more combinatorial in nature and as such a somewhat harder challenge than the positive counterpart.

\vspace*{0.25in} \noindent {\bf Index Terms: Negative spherical perceptron; storage capacity}.

\end{singlespace}

\section{Introduction}
\label{sec:back}

We start by briefly revisiting the basics of the perceptron problems. These problems have its roots in a variety of applications, ranging from standard neural networks to statistical physics and even bio-engineering. Since, our main interest will be the underlying mathematics of perceptron models we will start by recalling on a few mathematical definitions/descriptions typically used to introduce a perceptron (for the context that we will need here a bit more detailed presentation on how a perceptron works can be found e.g. in our recent work \cite{StojnicGardGen13}; on the other hand a way more detailed presentation of a perceptron's functioning can be found in many excellent references, see e.g. \cite{Gar88,GarDer88,AmiGutSom85,AmiGutSom87,TalBook}). Before introducing the details we also mention that there are many variants of the perceptrons (some of them have been studied in \cite{Gar88,GarDer88,AmiGutSom85,AmiGutSom87,TalBook}). The one that we will study here is typically called the spherical perceptron (or often alternatively the Gardner problem as the first analytical description of its predicted functioning was presented in \cite{Gar88}).

Mathematically, the description of the spherical perceptron (or for that matter of pretty much any perceptron) typically starts with the following dynamics:
\begin{equation}
H_{ik}^{(t+1)}=\mbox{sign} (\sum_{j=1,j\neq k}^{n}H_{ij}^{(t)}X_{jk}-T_{ik}).\label{eq:defdyn}
\end{equation}
Following \cite{Gar88} for any fixed $1\leq i\leq m$ we will call each $H_{ij},1 \leq j\leq n $, the icing spin, i.e. $H_{ij}\in\{-1,1\},\forall i,j$. Following \cite{Gar88} further we will call $X_{jk},1\leq j\leq n$, the interaction strength for the bond from site $j$ to site $i$. $T_{ik},1\leq i\leq m,1\leq k\leq n$, will be the threshold for site $k$ in pattern $i$ (we will typically assume that $T_{ik}=0$; however all the results we present below can be modified easily so that they include scenarios where $T_{ik}\neq 0$).

Now, the dynamics presented in (\ref{eq:defdyn}) works by moving from a $t$ to $t+1$ and so on (of course one assumes an initial configuration for say $t=0$). Moreover, the above dynamics will have a fixed point if say there are strengths $X_{jk},1\leq j\leq n,1\leq k\leq m$, such that for any $1\leq i\leq m$
\begin{eqnarray}
& & H_{ik}\mbox{sign} (\sum_{j=1,j\neq k}^{n}H_{ij}X_{jk}-T_{ik})=1\nonumber \\
& \Leftrightarrow & H_{ik}(\sum_{j=1,j\neq k}^{n}H_{ij}X_{jk}-T_{ik})>0,1\leq j\leq n,1\leq k\leq n.\label{eq:defdynfp}
\end{eqnarray}
Now, of course this is a well known property of a very general class of dynamics. In other words, unless one specifies the interaction strengths the generality of the problem essentially makes it easy. In \cite{Gar88} then proceeded and considered the spherical restrictions on $X$. To be more specific the restrictions considered in \cite{Gar88} amount to the following constraints
\begin{equation}
\sum_{j=1}^{n}X_{ji}^2=1,1\leq i\leq n.\label{eq:cosntX}
\end{equation}
There are many interesting questions one can ask about the above dynamics. Here, we will be interested in a specific one, namely the so-called storage capacity.
When it comes to storage capacity, one basically asks how many patterns $m$ ($i$-th pattern being $H_{ij},1\leq j\leq n$) one can store so that there is an assurance that they are stored in a stable way. Moreover, since having patterns being fixed points of the above introduced dynamics is not enough to insure having a finite basin of attraction one often may impose a bit stronger threshold condition
\begin{eqnarray}
& & H_{ik}\mbox{sign} (\sum_{j=1,j\neq k}^{n}H_{ij}X_{jk}-T_{ik})=1\nonumber \\
& \Leftrightarrow & H_{ik}(\sum_{j=1,j\neq k}^{n}H_{ij}X_{jk}-T_{ik})>\kappa,1\leq j\leq n,1\leq k\leq n,\label{eq:defdynfpstr}
\end{eqnarray}
where typically $\kappa$ is a positive number. We will refer to a perceptron governed by the above dynamics and coupled with the spherical restrictions and a positive threshold $\kappa$ as the positive spherical perceptron. Alternatively, when $\kappa$ is negative we will refer to it as the negative spherical perceptron.

Also, we should mentioned that many variants of the model that we study here are possible from a purely mathematical perspective. However, many of them have found applications in various other fields as well. For example, a great set of references that contains a collection of results related to various aspects of neural networks and their bio-applications is  \cite{AgiAnnBarCooTan13a,AgiAnnBarCooTan13b,AgiBarBarGalGueMoa12,AgiBarGalGueMoa12,AgiAstBarBurUgu12}.

Before proceeding further, we briefly mention how the rest of the paper is organized. In Section \ref{sec:knownres} we will present several results that are known for the positive spherical perceptron. In Section \ref{sec:negkappa} we will present the results that can be obtained for the negative spherical perceptron using the methodology that is enough to settle the positive case completely. In Section \ref{sec:neglowstorcap} we will then present a mechanism that can be used to potentially lower the values of the storage capacity of the negative spherical perceptron typically obtained by solely relying on the strategy that is successful in the positive case. In Section \ref{sec:conc} we will discuss obtained results and present several concluding remarks.

\section{Known results}
\label{sec:knownres}

Our main interest in this paper will be understanding of underlying mathematics of the negative spherical perceptron. However, since from a neural networks point of view the positive spherical perceptron is a more realistic (needed) scenario below we will briefly sketch several known results that relate to its storage capacity. In \cite{Gar88} a replica type of approach was designed and based on it a characterization of the storage capacity was presented. Before showing what exactly such a characterization looks like we will first formally define it. Namely, throughout the paper we will assume the so-called linear regime, i.e. we will consider the so-called \emph{linear} scenario where the length and the number of different patterns, $n$ and $m$, respectively are large but proportional to each other. Moreover, we will denote the proportionality ratio by $\alpha$ (where $\alpha$ obviously is a constant independent of $n$) and will set
\begin{equation}
m=\alpha n.\label{eq:defmnalpha}
\end{equation}
Now, assuming that $H_{ij},1\leq i\leq m,1\leq j\leq n$, are i.i.d. symmetric Bernoulli random variables, \cite{Gar88} using the replica approach gave the following estimate for $\alpha$ so that (\ref{eq:defdynfpstr}) holds with overwhelming probability (under overwhelming probability we will in this paper assume a probability that is no more than a number exponentially decaying in $n$ away from $1$)
\begin{equation}
\alpha_c(\kappa)=(\frac{1}{\sqrt{2\pi}}\int_{-\kappa}^{\infty}(z+\kappa)^2e^{-\frac{z^2}{2}}dz)^{-1}.\label{eq:garstorcap}
\end{equation}
Based on the above characterization one then has that $\alpha_c$ achieves its maximum over positive $\kappa$'s as $\kappa\rightarrow 0$. One in fact easily then has
\begin{equation}
\lim_{\kappa\rightarrow 0}\alpha_c(\kappa)=2.\label{eq:garstorcapk0}
\end{equation}
The result given in (\ref{eq:garstorcapk0}) is of course well known and has been rigorously established either as a pure mathematical fact or even in the context of neural networks and pattern recognition \cite{Schlafli,Cover65,Winder,Winder61,Wendel62,Cameron60,Joseph60,BalVen87,Ven86}. In a more recent work \cite{SchTir02,SchTir03,TalBook} the authors also considered the Gardner problem and established that (\ref{eq:garstorcap}) also holds. In our own work \cite{StojnicGardGen13} we also considered the Gardner problem and presented an alternative mathematical approach that was powerful enough to reestablish the storage capacity prediction given in (\ref{eq:garstorcap}). We will below formalize the results obtained in \cite{SchTir02,SchTir03,TalBook,StojnicGardGen13}.

\begin{theorem} \cite{SchTir02,SchTir03,TalBook,StojnicGardGen13}
Let $H$ be an $m\times n$ matrix with $\{-1,1\}$ i.i.d.Bernoulli components. Let $n$ be large and let $m=\alpha n$, where $\alpha>0$ is a constant independent of $n$. Let $\alpha_c$ be as in (\ref{eq:garstorcap}) and let $\kappa\geq 0$ be a scalar constant independent of $n$. If $\alpha>\alpha_c$ then with overwhelming probability there will be no $\x$ such that $\|\x\|_2=1$ and (\ref{eq:defdynfpstr}) is feasible. On the other hand, if $\alpha<\alpha_c$ then with overwhelming probability there will be an $\x$ such that $\|\x\|_2=1$ and (\ref{eq:defdynfpstr}) is feasible.
\label{thm:SchTirTalSto}
\end{theorem}
\begin{proof}
Presented in \cite{SchTir02,SchTir03,TalBook,StojnicGardGen13}.
\end{proof}

As mentioned earlier, the results given in the above theorem essentially settle the storage capacity of the positive spherical perceptron or the Gardner problem. However, there are a couple of facts that should be pointed out (emphasized):

1) The results presented above relate to the \emph{positive} spherical perceptron. It is not clear at all if they would automatically translate to the case of the negative spherical perceptron. As mentioned earlier, the case of the negative spherical perceptron may be more of interest from a purely mathematical point of view than it is from say the neural networks point of view. Nevertheless, such a mathematical problem may turn out to be a bit harder than the standard positive case. In fact, in \cite{TalBook}, Talagrand conjectured (conjecture 8.4.4) that the above mentioned $\alpha_c$ remains an upper bound on the storage capacity even when $\kappa<0$, i.e. even in the case of the negative spherical perceptron. However, he does seem to leave it as an open problem what the exact value of the storage capacity in the negative case should be. In our own work \cite{StojnicGardGen13} we confirmed the Talagrand conjecture and showed that even in the negative case $\alpha_c$ from (\ref{eq:garstorcap}) is indeed an upper bound on the storage capacity.

2) It is rather clear but we do mention that the overwhelming probability statement in the above theorem is taken with respect to the randomness of $H$. To analyze the feasibility of (\ref{eq:defprobucor1}) we in \cite{StojnicGardGen13} relied on a mechanism we recently developed for studying various optimization problems in \cite{StojnicRegRndDlt10}. Such a mechanism works for various types of randomness. However, the easiest way to present it was assuming that the underlying randomness is standard normal. So to fit the feasibility of (\ref{eq:defprobucor1}) into the framework of \cite{StojnicRegRndDlt10} we in \cite{StojnicGardGen13} formally assumed that the elements of matrix $H$ are i.i.d. standard normals. In that regard then what was proved in \cite{StojnicGardGen13} is a bit different from what was stated in the above theorem. However, as mentioned in \cite{StojnicGardGen13} (and in more detail in \cite{StojnicRegRndDlt10,StojnicMoreSophHopBnds10}) all our results from \cite{StojnicGardGen13} continue to hold for a very large set of types of randomness and certainly for the Bernouilli one assumed in Theorem \ref{thm:SchTirTalSto}.

3) We will continue to call the critical value of $\alpha$ so that (\ref{eq:defdynfpstr}) is feasible the storage capacity even when $\kappa<0$, even though it may be linguistically a bit incorrect, given the neural network interpretation of finite basins of attraction mentioned above.

4) We should also mention that the results presented in Theorem \ref{thm:SchTirTalSto} relate to what is typically known as the uncorrelated case of the (positive) spherical perceptron. A fairly important alternative is the so-called correlated case, already considered in the introductory paper \cite{Gar88} as well as in many references that followed (see e.g. \cite{GutSte90}). As we have shown in \cite{StojnicGardGen13} one can create an alternative version of the results presented in Theorem \ref{thm:SchTirTalSto} that would correspond to such a case. However, since in this paper we will focus mostly on the standard, i.e. uncorrelated case, we will skip recalling on the corresponding correlated results presented in \cite{StojnicGardGen13}. We do mention though, that all major results that we will present in this paper can be translated to the correlated case as well. Since the derivations in the correlated case are a bit more involved we will presented them in a separate paper.

\section{Negative $\kappa$}
\label{sec:negkappa}

Now, we will proceed forward by actually heavily relying on the above mentioned points 1), 2), and 3). Namely, as far as point 2) goes, we will in this paper without loss of generality again make the same type of assumption that we have made in \cite{StojnicGardGen13} related to the statistics of $H$. Innother words, we will continue to assume that the elements of matrix $H$ are i.i.d. standard normals (as mentioned above, such an assumption changes nothing in the validity of the results that we will present; also, more on this topic can be found in e.g. \cite{StojnicHopBnds10,StojnicLiftStrSec13,StojnicRegRndDlt10} where we discussed it a bit further). On the other hand, as far as points 1) and 3) go, we will below in a theorem summarize what was actually proved in \cite{StojnicGardGen13} (this will also substantially facilitate the exposition that will follow). However, before doing so, we will briefly recall on a few simplifications that we utilized when creating results presented in \cite{StojnicGardGen13}.

As mentioned above under point 4), we have looked in \cite{StojnicGardGen13} and will look in this paper at a variant of the spherical perceptron (or the Gardner problem) that is typically called uncorrelated. In the uncorrelated case, one views all patterns $H_{i,1:n},1\leq i\leq m$, as uncorrelated (as expected, $H_{i,1:n}$ stands for vector $[H_{i1},H_{i2},\dots,H_{in}]$). Now, the following becomes the corresponding version of the question of interest mentioned above: assuming that $H$ is an $m\times n$ matrix with i.i.d. $\{-1,1\}$ Bernoulli entries and that $\|\x\|_2=1$, how large $\alpha=\frac{m}{n}$ can be so that the following system of linear inequalities is satisfied with overwhelming probability
\begin{equation}
H\x\geq \kappa.\label{eq:defprobucor}
\end{equation}
This of course is the same as if one asks how large $\alpha$ can be so that the following optimization problem is feasible with overwhelming probability
\begin{eqnarray}
& & H\x\geq \kappa\nonumber \\
& & \|\x\|_2=1.\label{eq:defprobucor1}
\end{eqnarray}
To see that (\ref{eq:defprobucor}) and (\ref{eq:defprobucor1}) indeed match the above described fixed point condition it is enough to observe that due to statistical symmetry one can assume $H_{i1}=1,1\leq i\leq m$. Also the constraints essentially decouple over the columns of $X$ (so one can then think of $\x$ in (\ref{eq:defprobucor}) and (\ref{eq:defprobucor1}) as one of the columns of $X$). Moreover, the dimension of $H$ in (\ref{eq:defprobucor}) and (\ref{eq:defprobucor1}) should be changed to $m\times (n-1)$; however, since we will consider a large $n$ scenario to make writing easier we keep dimension as $m\times n$. Also, as mentioned under point 2) above, we will, without a loss of generality, treat $H$ in (\ref{eq:defprobucor1}) as if it has i.i.d. standard normal components. Moreover, in \cite{StojnicGardGen13} we also recognized that (\ref{eq:defprobucor1}) can be rewritten as the following optimization problem
\begin{eqnarray}
\xi_n=\min_{\x} \max_{\lambda\geq 0} & &  \kappa\lambda^T\1- \lambda^T H\x \nonumber \\
\mbox{subject to} & & \|\lambda\|_2= 1\nonumber \\
& & \|\x\|_2=1,\label{eq:uncorminmax}
\end{eqnarray}
where $\1$ is an $m$-dimensional column vector of all $1$'s. Clearly, if $\xi_n\leq 0$ then (\ref{eq:defprobucor1}) is feasible. On the other hand, if $\xi_n>0$ then (\ref{eq:defprobucor1}) is not feasible. That basically means that if we can probabilistically characterize the sign of $\xi_n$ then we could have a way of determining $\alpha$ such that $\xi_n\leq 0$. That is exactly what we have done in \cite{StojnicGardGen13} on an ultimate level for $\kappa\geq 0$ and on a say upper-bounding level for $\kappa<0$. Relying on the strategy developed in \cite{StojnicRegRndDlt10,StojnicGorEx10} and on a set of results from \cite{Gordon85,Gordon88} we in \cite{StojnicGardGen13} proved the following theorem:

\begin{theorem} \cite{StojnicGardGen13}
Let $H$ be an $m\times n$ matrix with $\{-1,1\}$ i.i.d. standard normal components. Let $n$ be large and let $m=\alpha n$, where $\alpha>0$ is a constant independent of $n$. Let $\xi_n$ be as in (\ref{eq:uncorminmax}) and let $\kappa$ be a scalar constant independent of $n$. Let all $\epsilon$'s be arbitrarily small constants independent of $n$. Further, let $\g_i$ be a standard normal random variable and set
\begin{equation}
f_{gar}(\kappa)=\frac{1}{\sqrt{2\pi}}\int_{-\kappa}^{\infty}(\g_i+\kappa)^2e^{-\frac{\g_i^2}{2}}d\g_i.\label{eq:fgarlemmaunncorlb}
\end{equation}
Let $\xi_n^{(l)}$ and $\xi_n^{(u)}$ be scalars such that
\begin{eqnarray}
(1-\epsilon_{1}^{(m)})\sqrt{\alpha f_{gar}(\kappa)}-(1+\epsilon_{1}^{(n)})-\epsilon_{5}^{(g)} & > & \frac{\xi_n^{(l)}}{\sqrt{n}}\nonumber \\
(1+\epsilon_{1}^{(m)})\sqrt{\alpha f_{gar}(\kappa)}-(1-\epsilon_{1}^{(n)})+\epsilon_{5}^{(g)} & < & \frac{\xi_n^{(u)}}{\sqrt{n}}.\label{eq:condxinthmstoc30}
\end{eqnarray}
If $\kappa\geq 0$ then
\begin{equation}
 \lim_{n\rightarrow\infty}P(\xi_n^{(l)}\leq \xi_n\leq \xi_n^{(u)})=\lim_{n\rightarrow\infty}P(\min_{\|\x\|_2=1}\max_{\|\lambda\|_2=1,\lambda_i\geq 0}(\xi_n^{(l)}\leq \kappa\lambda^T\1-\lambda^TH\x)\leq \xi_n^{(u)})\geq 1. \label{eq:probthmstoc30poskappa}
\end{equation}
Moreover, if $\kappa< 0$ then
\begin{equation}
 \lim_{n\rightarrow\infty}P(\xi_n\geq \xi_n^{(l)})=\lim_{n\rightarrow\infty}P(\min_{\|\x\|_2=1}\max_{\|\lambda\|_2=1,\lambda_i\geq 0}( \kappa\lambda^T\1-\lambda^TH\x)\geq \xi_n^{(u)})\geq 1. \label{eq:probthmstoc30negkappa}
\end{equation}
\label{thm:Stoc30}
\end{theorem}
\begin{proof}
Presented in \cite{StojnicGardGen13}.
\end{proof}
In a more informal language (essentially ignoring all technicalities and $\epsilon$'s) one has that as long as
\begin{equation}
\alpha>\frac{1}{f_{gar}(\kappa)},\label{eq:condalphauncorlb}
\end{equation}
the problem in (\ref{eq:defprobucor1}) will be infeasible with overwhelming probability. On the other hand, one has that when $\kappa\geq 0$ as long as
\begin{equation}
\alpha<\frac{1}{f_{gar}(\kappa)},\label{eq:condalphauncorubpos}
\end{equation}
the problem in (\ref{eq:defprobucor1}) will be feasible with overwhelming probability. This of course settles the case $\kappa geq 0$ completely and essentially establishes the storage capacity as $\alpha_c$ which of course matches the prediction given in the introductory analysis given in \cite{Gar88} and of course rigorously confirmed by the results of \cite{SchTir02,SchTir03,TalBook}. On the other hand, when $\kappa < 0$ it only shows that the storage capacity with overwhelming probability is not higher than  As mentioned earlier this confirms Talagrand's conjecture 8.4.4 from \cite{TalBook}. However, it does not settle problem (question) 8.4.2 from \cite{TalBook}. Moreover, it is not even clear if the critical storage capacity when $\kappa<0$ in fact has anything to do with the above mentioned quantity $\frac{1}{f_{gar}(\kappa)}$ beyond that it is not larger than it.

The results obtained based on the above theorem are presented in Figure \ref{fig:alfackappa}. When $\kappa\geq 0$ (i.e. when $\alpha\leq 2$) the curve indicates the exact breaking point between the ``overwhelming" feasibility and infeasibility of (\ref{eq:defprobucor1}). On the other hand, when $\kappa< 0$ (i.e. when $\alpha> 2$) the curve is only an upper bound on the storage capacity, i.e. for any value of the pair $(\alpha,\kappa)$ that is above the curve given in Figure \ref{fig:alfackappa}, (\ref{eq:defprobucor1}) is infeasible with overwhelming probability.
\begin{figure}[htb]
\centering
\centerline{\epsfig{figure=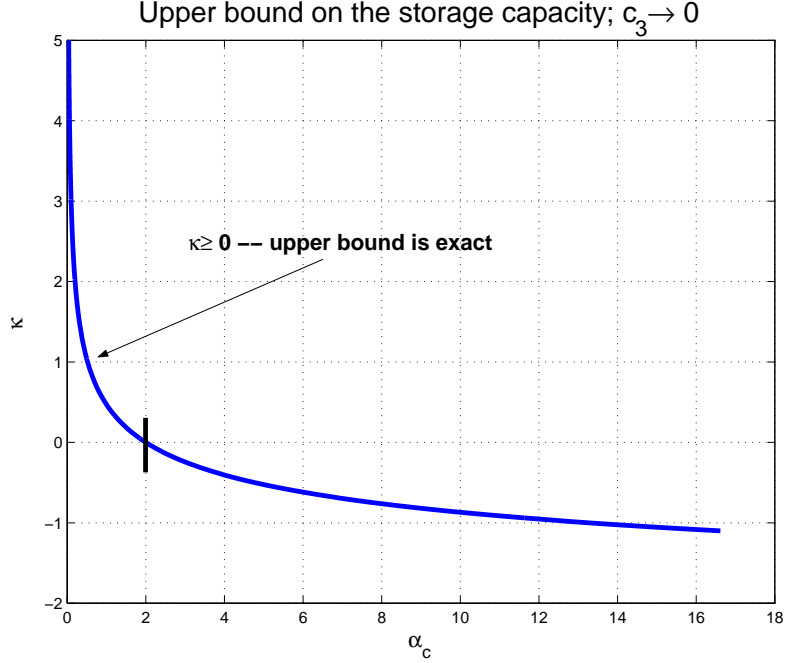,width=10.5cm,height=9cm}}
\caption{$\alpha_c$ as a function of $\kappa$}
\label{fig:alfackappa}
\end{figure}

Below, we present a fairly powerful mechanism that can be used to attack the negative spherical perceptron case beyond what is presented above. The mechanism will rely on a progress we have made recently in a series of papers \cite{StojnicMoreSophHopBnds10,StojnicLiftStrSec13,StojnicLiftLqThrBnds13,StojnicAsymmLittBnds11} when studying various other combinatorial problems. It will essentially attempt to lower the above upper bound on the storage capacity. We should mention though, that when it comes to the spherical perceptron the value of the mechanism that we will present seems more conceptual than practical. Improving on the results that we presented in Theorem \ref{thm:Stoc30} (and on those we presented in \cite{StojnicGardGen13}) does not seem as an easy task (if possible at all) and after studying it for some time we arrived at a point where we view any potential avenue for such an improvement as a nice breakthrough. The preliminary numerical investigation (the results of which we will present below) indicates that one may indeed be able to lower the above value of the storage capacity when $\kappa<0$.

\section{Lowering the storage capacity when $\kappa<0$}
\label{sec:neglowstorcap}

As mentioned above, in this section we look at a strategy that can potentially be used to lower the above upper bound on the storage capacity implied by (\ref{eq:condalphauncorlb}). Before proceeding further we will first recall on the optimization problem that we will consider here. It is basically the one given in (\ref{eq:uncorminmax})
\begin{eqnarray}
\xi_n=\min_{\x} \max_{\lambda\geq 0} & &  \kappa\lambda^T\1- \lambda^T H\x \nonumber \\
\mbox{subject to} & & \|\lambda\|_2= 1\nonumber \\
& & \|\x\|_2=1,\label{eq:uncorminmax1}
\end{eqnarray}
where $\1$ is an $m$-dimensional column vector of all $1$'s. As mentioned below (\ref{eq:uncorminmax1}), a probabilistic characterization of the sign of $\xi_n$ would be enough to determine the storage capacity or its bounds. Below, we provide a way that can be used to probabilistically characterize $\xi_n$. Moreover, since $\xi_n$ will concentrate around its mean for our purposes here it will then be enough to study only its mean $E\xi_n$. We do so by relying on the strategy developed in \cite{StojnicMoreSophHopBnds10} and ultimately on the following set of results from \cite{Gordon85} (the following theorem presented in \cite{StojnicMoreSophHopBnds10} is in fact a slight alternation of the original results from \cite{Gordon85}).
\begin{theorem}(\cite{Gordon85})
\label{thm:Gordonneg1} Let $X_{ij}$ and $Y_{ij}$, $1\leq i\leq n,1\leq j\leq m$, be two centered Gaussian processes which satisfy the following inequalities for all choices of indices
\begin{enumerate}
\item $E(X_{ij}^2)=E(Y_{ij}^2)$
\item $E(X_{ij}X_{ik})\geq E(Y_{ij}Y_{ik})$
\item $E(X_{ij}X_{lk})\leq E(Y_{ij}Y_{lk}), i\neq l$.
\end{enumerate}
Let $\psi_{ij}()$ be increasing functions on the real axis. Then
\begin{equation*}
E(\min_{i}\max_{j}\psi_{ij}(X_{ij}))\leq E(\min_{i}\max_{j}\psi_{ij}(Y_{ij})).
\end{equation*}
Moreover, let $\psi_{ij}()$ be decreasing functions on the real axis. Then
\begin{equation*}
E(\max_{i}\min_{j}\psi_{ij}(X_{ij}))\geq E(\max_{i}\min_{j}\psi_{ij}(Y_{ij})).
\end{equation*}
\begin{proof}
The proof of all statements but the last one is of course given in \cite{Gordon85}. The proof of the last statement trivially follows and in a slightly different scenario is given for completeness in \cite{StojnicMoreSophHopBnds10}.
\end{proof}
\end{theorem}
The strategy that we will present below will utilize the above theorem to lift the above mentioned lower bound on $\xi_n$ (of course since we talk in probabilistic terms, under bound on $\xi_n$ we essentially assume a bound on $E\xi_n$). We do mention that in \cite{StojnicGardGen13} we relied on a variant of the above theorem to create a probabilistic lower bound on $\xi_n$. However, the strategy employed in \cite{StojnicGardGen13} relied only on a basic version of the above theorem which assumes $\psi_{ij}(x)=x$. Here, we will substantially upgrade the strategy from \cite{StojnicGardGen13} by looking at a very simple (but way better) different version of $\psi_{ij}()$.

\subsection{Lifting lower bound on $\xi_n$}
\label{sec:uncorgardlb}

In \cite{StojnicMoreSophHopBnds10} we established a lemma very similar to the following one:
\begin{lemma}
Let $A$ be an $m\times n$ matrix with i.i.d. standard normal components. Let $\g$ and $\h$ be $m\times 1$ and $n\times 1$ vectors, respectively, with i.i.d. standard normal components. Also, let $g$ be a standard normal random variable and let $c_3$ be a positive constant. Then
\begin{equation}
E(\max_{\|\x\|_2=1}\min_{\|\lambda\|_2=1,\lambda_i\geq 0}e^{-c_3(-\lambda^T H\x + g +\kappa\lambda^T\1)})\leq E(\max_{\x}\min_{\|\lambda\|_2=1,\lambda_1\geq 0}e^{-c_3(\g^T\lambda+\h^T\x+\kappa\lambda^T\1)}).\label{eq:negexplemma}
\end{equation}\label{lemma:negexplemma}
\end{lemma}
\begin{proof}
As mentioned in \cite{StojnicMoreSophHopBnds10}, the proof is a standard/direct application of Theorem \ref{thm:Gordonneg1}. We will omit the details since they are pretty much the same as the those in the proof of the corresponding lemma in \cite{StojnicMoreSophHopBnds10}. However, we do mention that the only difference between this lemma and the one in \cite{StojnicMoreSophHopBnds10} is in the structure of the sets of allowed values for $\x$ and $\lambda$. However, such a difference introduces no structural changes in the proof.
\end{proof}

Following step by step what was done after Lemma 3 in \cite{StojnicMoreSophHopBnds10} one arrives at the following analogue of \cite{StojnicMoreSophHopBnds10}'s equation $(57)$:
\begin{equation}
\hspace{-.5in}E(\min_{\|\x\|_2=1}\max_{\|\lambda\|_2=1,\lambda_i\geq 0}(-\lambda^T H\x+\kappa\lambda^T\1))\geq
\frac{c_3}{2}-\frac{1}{c_3}\log(E(\max_{\|\x\|_2=1}(e^{-c_3\h^T\x})))
-\frac{1}{c_3}\log(E(\min_{\|\lambda\|_2=1,\lambda_i\geq 0}(e^{-c_3(\g^T\lambda+\kappa\lambda^T\1)}))).\label{eq:chneg8}
\end{equation}
Let $c_3=c_3^{(s)}\sqrt{n}$ where $c_3^{(s)}$ is a constant independent of $n$. Then (\ref{eq:chneg8}) becomes
\begin{eqnarray}
\hspace{-.5in}\frac{E(\min_{\|\x\|_2=1}\max_{\|\lambda\|_2=1,\lambda_i\geq 0}(-\lambda^T H\x+\kappa\lambda^T\1))}{\sqrt{n}}
& \geq &
\frac{c_3^{(s)}}{2}-\frac{1}{nc_3^{(s)}}\log(E(\max_{\|\x\|_2=1}(e^{-c_3^{(s)}\sqrt{n}\h^T\x})))\nonumber \\
& - & \frac{1}{nc_3^{(s)}}\log(E(\min_{\|\lambda\|_2=1,\lambda_i\geq 0}(e^{-c_3^{(s)}\sqrt{n}(\g^T\lambda+\kappa\lambda^T\1)})))\nonumber \\
& = &-(-\frac{c_3^{(s)}}{2}+I_{sph}(c_3^{(s)})+I_{per}(c_3^{(s)},\alpha,\kappa)),\label{eq:chneg9}
\end{eqnarray}
where
\begin{eqnarray}
I_{sph}(c_3^{(s)}) & = & \frac{1}{nc_3^{(s)}}\log(E(\max_{\|\x\|_2=1}(e^{-c_3^{(s)}\sqrt{n}\h^T\x})))\nonumber \\
I_{per}(c_3^{(s)},\alpha,\kappa) & = & \frac{1}{nc_3^{(s)}}\log(E(\min_{\|\lambda\|_2=1,\lambda_i\geq 0}(e^{-c_3^{(s)}\sqrt{n}(\g^T\lambda+\kappa\lambda^T\1)}))).\label{eq:defIs}
\end{eqnarray}
Moreover, \cite{StojnicMoreSophHopBnds10} also established
\begin{equation}
\hspace{-.5in}I_{sph}(c_3^{(s)}) = \frac{1}{nc_3^{(s)}}\log(E(\max_{\|\x\|_2=1}(e^{-c_3^{(s)}\sqrt{n}\h^T\x})))
 \doteq\widehat{\gamma^{(s)}}-\frac{1}{2c_3^{(s)}}\log(1-\frac{c_3^{(s)}}{2\widehat{\gamma^{(s)}}}),\label{eq:ubmorsoph}
\end{equation}
where
\begin{equation}
\widehat{\gamma^{(s)}}=\frac{2c_3^{(s)}+\sqrt{4(c_3^{(s)})^2+16}}{8},\label{eq:gamaiden3}
\end{equation}
and $\doteq$ stands for equality when $n\rightarrow\infty$ (as mentioned in \cite{StojnicMoreSophHopBnds10}, $\doteq$ in (\ref{eq:ubmorsoph}) is exactly what was shown in \cite{SPH}.

To be able to use the bound in (\ref{eq:chneg8}) we would also need a characterization of $I_{per}(c_3^{(s)},\alpha,\kappa)$. Below we provide a way to characterize $I_{per}(c_3^{(s)},\alpha,\kappa)$. Let $f(\lambda)=-\g^T\lambda-\kappa\lambda^T\1$ and
we start with the following line of identities
\begin{multline}
\min_{\|\lambda\|_2=1,\lambda_i\geq 0}f(\lambda;\g,\kappa)=\min_{\|\lambda\|_2=1,\lambda_i\geq 0}(-\g^T\lambda-\kappa\lambda^T\1)
=\min_{\lambda_i\geq 0}\max_{\gamma_{per}\geq 0} -\g^T\lambda-\kappa\lambda^T\1
+\gamma_{per}\sum_{i=1}^{m}\lambda_i^2-\gamma_{per}\\
=\max_{\gamma_{per}\geq 0}\min_{\lambda_i\geq 0} -\g^T\lambda-\kappa\lambda^T\1+\gamma_{per}\sum_{i=1}^{m}\lambda_i^2-\gamma_{per}=\max_{\gamma_{per}\geq 0} -\frac{1}{4\gamma_{per}}\left (\sum_{i=1}^{m}(\max(\g_i+\kappa,0))^2\right )-\gamma_{per}\\
=\max_{\gamma_{per}\geq 0} \frac{f_1(\g,\kappa)}{4\gamma_{per}}+\gamma_{per},\label{eq:seceq1}
\end{multline}
where
\begin{equation}
f_1(\g,\kappa)=\left (\sum_{i=1}^{m}(\max(\g_i+\kappa,0))^2\right ).\label{eq:deff1}
\end{equation}
Then
\begin{multline}
\hspace{-.3in}I_{per}(c_3^{(s)},\alpha,\kappa)  =  \frac{1}{nc_3^{(s)}}\log(E(\min_{\|\lambda\|_2=1,\lambda_i\geq 0}(e^{-c_3^{(s)}\sqrt{n}(\g^T\lambda+\kappa\lambda^T\1)}))) = \frac{1}{nc_3^{(s)}}\log(E(\min_{\|\lambda\|_2=1,\lambda_i\geq 0}(e^{c_3^{(s)}\sqrt{n}f(\lambda;\g,\kappa))})))\\=\frac{1}{nc_3^{(s)}}\log(Ee^{c_3^{(s)}\sqrt{n}\max_{\gamma_{per}\geq 0}(\frac{f_1(\g,\kappa)}{4\gamma_{per}}+\gamma_{per})})
\doteq \frac{1}{nc_3^{(s)}}\max_{\gamma_{per}\geq 0}\log(Ee^{c_3^{(s)}\sqrt{n}(\frac{f_1(\g,\kappa)}{4\gamma_{per}}+\gamma_{per})})\\
=\max_{\gamma_{per}\geq 0}(\frac{\gamma_{per}}{\sqrt{n}}+\frac{1}{nc_3^{(s)}}\log(Ee^{c_3^{(s)}\sqrt{n}(\frac{f_1(\g,\kappa)}{4\gamma_{per}})})),\label{eq:gamaiden1sec}
\end{multline}
where, as earlier, $\doteq$ stands for equality when $n\rightarrow \infty$ and would be obtained through the mechanism presented in \cite{SPH}. Now if one sets $\gamma_{per}=\gamma_{per}^{(s)}\sqrt{n}$ then (\ref{eq:gamaiden1sec}) gives
\begin{multline}
I_{per}(c_3^{(s)},\alpha,\kappa)
\doteq\max_{\gamma_{per}\geq 0}(\frac{\gamma_{per}}{\sqrt{n}}+\frac{1}{nc_3^{(s)}}\log(Ee^{c_3^{(s)}\sqrt{n}(\frac{f_1(\g,\kappa)}{4\gamma_{per}})})\\
=\max_{\gamma_{per}^{(s)}\geq 0}(\gamma_{per}^{(s)}+\frac{\alpha}{c_3^{(s)}}\log(Ee^{(\frac{c_3^{(s)}(\max(\g_i+\kappa,0))^2}{4\gamma_{per}^{(s)}})})
=\max_{\gamma_{per}^{(s)}\geq 0}(\gamma_{per}^{(s)}+\frac{\alpha}{c_3^{(s)}}\log(I_{per}^{(1)}(c_3^{(s)},\gamma_{per}^{(s)},\kappa))),\label{eq:gamaiden2sec}
\end{multline}
where
\begin{equation}
I_{per}^{(1)}(c_3^{(s)},\gamma_{per}^{(s)},\kappa) = Ee^{(\frac{c_3^{(s)}(\max(\g_i+\kappa,0))^2}{4\gamma_{per}^{(s)}})}.\label{eq:defI1I2sec}
\end{equation}
A combination of (\ref{eq:gamaiden2sec}) and (\ref{eq:defI1I2sec}) is then enough to enable us to use the bound in (\ref{eq:chneg8}). However, one can be a bit more explicit when it comes to $I_{per}^{(1)}(c_3^{(s)},\gamma_{per}^{(s)},\kappa)$. Set
\begin{eqnarray}
p & = & 1+\frac{c_3^{(s)}}{2\gamma_{per}^{(s)}}\nonumber \\
q & = & \frac{c_3^{(s)}\kappa}{2\gamma_{per}^{(s)}}\nonumber \\
r & = & \frac{c_3^{(s)}\kappa^2}{4\gamma_{per}^{(s)}}\nonumber \\
s & = & -\kappa\sqrt{p}+\frac{q}{\sqrt{p}} \nonumber \\
C & = & \frac{exp(\frac{q^2}{2p}-r)}{\sqrt{p}}.\label{eq:helpdef}
\end{eqnarray}
Then
\begin{equation}
I_{per}^{(1)}(c_3^{(s)},\gamma_{per}^{(s)},\kappa)=\frac{1}{2}erfc(\frac{\kappa}{\sqrt{2}})+\frac{C}{2}(erfc(\frac{s}{\sqrt{2}})).\label{eq:Iper1}
\end{equation}

We summarize the results from this section in the following theorem.

\begin{theorem}
Let $H$ be an $m\times n$ matrix with $\{-1,1\}$ i.i.d. standard normal components. Let $n$ be large and let $m=\alpha n$, where $\alpha>0$ is a constant independent of $n$. Let $\xi_n$ be as in (\ref{eq:uncorminmax}) and let $\kappa<0$ be a scalar constant independent of $n$. Let all $\epsilon$'s be arbitrarily small constants independent of $n$. Further, let $\g_i$ be a standard normal random variable. Set
\begin{equation}
\widehat{\gamma^{(s)}}=\frac{2c_3^{(s)}+\sqrt{4(c_3^{(s)})^2+16}}{8},\label{eq:gamaliftthm}
\end{equation}
and
\begin{equation}
I_{sph}(c_3^{(s)}) = \widehat{\gamma^{(s)}}-\frac{1}{2c_3^{(s)}}\log(1-\frac{c_3^{(s)}}{2\widehat{\gamma^{(s)}}}).\label{eq:Isphthm}
\end{equation}
Let $I_{per}^{(1)}(c_3^{(s)},\gamma_{per}^{(s)},\kappa)$ be defined through (\ref{eq:helpdef}) and (\ref{eq:Iper1}). Set
\begin{equation}
I_{per}(c_3^{(s)},\alpha,\kappa)=\max_{\gamma_{per}^{(s)}\geq 0}(\gamma_{per}^{(s)}+\frac{1}{c_3^{(s)}}\log(I_{per}^{(1)}(c_3^{(s)},\gamma_{per}^{(s)},\kappa))).\label{eq:Iperthm}
\end{equation}
If $\alpha$ is such that
\begin{equation}
\min_{c_3^{(s)}\geq 0}(-\frac{c_3^{(s)}}{2}+I_{sph}(c_3^{(s)})+I_{per}(c_3^{(s)},\alpha,\kappa))<0,\label{eq:condliftsphnegthm}
\end{equation}
then (\ref{eq:defprobucor1}) is infeasible with overwhelming probability.
\label{thm:liftnegsphper}
\end{theorem}
\begin{proof}
Follows from the previous discussion by combining (\ref{eq:uncorminmax1}) and (\ref{eq:chneg9}), and  by noting that the bound given in (\ref{eq:chneg9}) holds for any $c_3^{(s)}\geq 0$ and could therefore be tightened by additionally optimizing over $c_3^{(s)}\geq 0$.
\end{proof}

The results one can obtain for the storage capacity based on the above theorem are presented in Figure \ref{fig:liftsphneg}. We should mention that the results presented in Figure \ref{fig:liftsphneg} should be taken only as an illustration. They are obtained as a result of a numerical optimization. Remaining finite precision errors are of course possible and could affect the validity of the obtained results (we do believe though, that the results presented in Figure \ref{fig:liftsphneg} are actually fairly close to what Theorem \ref{thm:liftnegsphper} predicts). Also, we should emphasize that the results presented in Theorem \ref{thm:liftnegsphper} are completely rigorous; it is just that some of the numerical work may have introduced a bit of imprecision.

Also, we found that a visible improvement in the values of the storage capacity (essentially a lowering of the upper bound values obtained in Theorem \ref{thm:Stoc30} and presented in Figure \ref{fig:alfackappa}) happens only for values of $\alpha$ substantially larger than $2$ (the results obtained based on Theorem \ref{thm:Stoc30} were presented in Figure \ref{fig:alfackappa}; we included them in Figure \ref{fig:liftsphneg} as well and denoted them as $c_3\rightarrow 0$ curve). That is of course the reason why we selected such a range for presenting the storage capacity upper bounds in Figure \ref{fig:liftsphneg}. Also, given the shape of the storage capacity bounds curves, even in such a range the improvement can not be easily seen in the given plot. For that reason we in Tables \ref{tab:liftsphnegtab1} and \ref{tab:liftsphnegtab2} give the concrete values we obtained for the storage capacity upper bounds (as well as those we obtained for parameters $c_3^{(s)}$ and $\gamma_{per}^{(s)}$) based on Theorem \ref{thm:liftnegsphper}. We denote by $\alpha_c^{(u,low)}$ the smallest $\alpha$ such that (\ref{eq:condliftsphnegthm}) holds. Along the same lines we denote by $\alpha_c^{(u)}$ the smallest $\alpha$ such that (\ref{eq:condalphauncorlb}) holds. In fact, $\alpha_c^{(u)}$ can also be obtained from Theorem \ref{thm:liftnegsphper} by taking $c_3^{(s)}\rightarrow 0$ and consequently $\gamma_{per}^{(s)}\rightarrow\frac{1}{2}$.

\begin{table}
\caption{Lowered upper bounds on $\alpha_c$ -- lower $\alpha$/higher $\kappa$ regime; optimized parameters}\vspace{.1in}
\hspace{-0in}\centering
\begin{tabular}{||l||c|c|c|c|c||}\hline\hline
 \hspace{.7in}$\kappa$                                                 & $-0.5$   & $-0.6$   & $-0.63$   & $-0.65$  & $-0.67$  \\ \hline\hline
 \hspace{.7in}$c_{3}^{(s)}$                                            & $0.0000$ & $0.0000$ & $0.0274$  & $0.0943$ & $0.1597$ \\ \hline
 \hspace{.7in}$\gamma_{per}^{(s)}$                                     & $0.5000$ & $0.5000$ & $0.4932$  & $0.4770$ & $0.4617$ \\ \hline
 \hspace{.7in}$\alpha_{c}^{(u,low)}$                                   & $\mathbf{4.7700}$ & $\mathbf{5.7787}$ & $\mathbf{6.12834}$ &
 $\mathbf{6.3737}$ & $\mathbf{6.6290}$ \\ \hline
 $\alpha_{c}^{(u)}$; $c_3^{(s)}\rightarrow 0$,
 $\gamma_{per}^{(s)}\rightarrow \frac{1}{2}$                           & $4.7700$ & $5.7787$ & $6.12847$ & $6.3754$ & $6.6339$ \\ \hline\hline
\end{tabular}
\label{tab:liftsphnegtab1}
\end{table}

\begin{table}
\caption{Lowered upper bounds on $\alpha_c$ -- higher $\alpha$/lower $\kappa$ regime; optimized parameters}\vspace{.1in}
\hspace{-0in}\centering
\begin{tabular}{||l||c|c|c|c|c||}\hline\hline
 \hspace{.7in}$\kappa$                                                 & $-0.7$   & $-0.8$   & $-0.9$    & $-1$      & $-1.1$  \\ \hline\hline
 \hspace{.7in}$c_{3}^{(s)}$                                            & $0.2555$ & $0.5591$ & $0.8470$  & $1.1266$  & $1.4029$ \\ \hline
 \hspace{.7in}$\gamma_{per}^{(s)}$                                     & $0.4402$ & $0.3794$ & $0.3312$  & $0.2922$  & $0.2600$ \\ \hline
 \hspace{.7in}$\alpha_{c}^{(u,low)}$                                   & $\mathbf{7.0313}$ & $\mathbf{8.5631}$ & $\mathbf{10.4484}$
 & $\mathbf{12.784}$  & $\mathbf{15.6977}$ \\ \hline
 $\alpha_{c}^{(u)}$; $c_3^{(s)}\rightarrow 0$,
 $\gamma_{per}^{(s)}\rightarrow \frac{1}{2}$                           & $7.0448$ & $8.6431$ & $10.6755$ & $13.2731$ & $16.6155$ \\ \hline\hline
\end{tabular}
\label{tab:liftsphnegtab2}
\end{table}

\begin{figure}[htb]
\centering
\centerline{\epsfig{figure=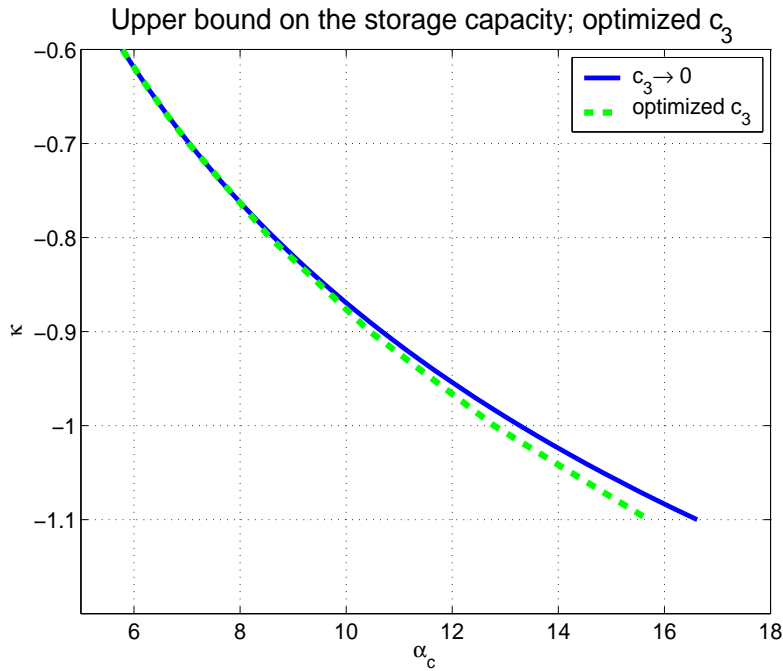,width=10.5cm,height=9cm}}
\caption{$\kappa$ as a function of $\alpha_{c}^{(u,low)}$}
\label{fig:liftsphneg}
\end{figure}

\section{Conclusion}
\label{sec:conc}

In this paper we revisited the storage capacity of the classical spherical perceptron problem. The storage capacity of the spherical perceptron has been settled in \cite{SchTir02,SchTir03} (and later reestablished in \cite{TalBook,StojnicGardGen13}). However, these results focused mostly on the so-called positive version of the corresponding mathematical problem. Here, we looked at the negative counterpart as well.

In our recent work \cite{StojnicGardGen13} we provided a mechanism that was powerful enough to reestablished the results for the storage capacity of the positive spherical perceptron. Moreover, we were able to use it to obtain the upper bounds on the storage capacity even in the case of negative spherical perceptron. These upper bounds are essentially of the same shape as the corresponding ones in the positive case (which in effect confirmed a conjecture of Talagrand given in \cite{TalBook}). Here, on the other hand, we designed a more powerful version of the mechanism given in \cite{StojnicGardGen13} that can be used to lower these upper bounds. In fact, as limited numerical calculations indicate the upper bounds obtained based on the mechanism presented in \cite{StojnicGardGen13} can indeed be lowered.

Of course, various other features of the spherical perceptron are also of interest. The framework that we presented here can be used to analyze pretty much all of them. However, here we focused only on the storage capacity as it is one of the most well known/studied ones. The results we presented though seem to indicate a phenomenon which to a degree applies to quite a few of these features. Namely, it is quite possible that the underlying mathematics describing the behavior of the spherical perceptron substantially changes as one moves from the positive (i.e. $\kappa\geq 0$) to the negative (i.e. $\kappa<0$) thresholding .

Also, as mentioned on several occasions throughout the paper, we in this paper focused on the standard uncorrelated version of the spherical perceptron. When it comes to the storage capacity it is often of quite an interest to study the correlated counterpart as well. Conceptually, it is not that hard to translate the results we presented here to the correlated case. However, the presentation of such a scenario is a bit more cumbersome and we will discuss it elsewhere.

Along the same lines, we should emphasize that in order to maintain the easiness of the exposition throughout the paper we presented a collection of theoretical results for a particular type of randomness, namely the standard normal one. However, as was the case when we studied the Hopfield models in \cite{StojnicHopBnds10,StojnicMoreSophHopBnds10} as well as the Gardner problem in \cite{StojnicGardGen13}, all results that we presented can easily be extended to cover a wide range of other types of randomness. There are many ways how this can be done (and the rigorous proofs are not that hard either). Typically they all would boil down to a repetitive use of the central limit theorem. For example, a particularly simple and elegant approach would be the one of Lindeberg \cite{Lindeberg22}. Adapting our exposition to fit into the framework of the Lindeberg principle is relatively easy and in fact if one uses the elegant approach of \cite{Chatterjee06} pretty much a routine. However, as we mentioned when studying the Hopfield model \cite{StojnicHopBnds10}, since we did not create these techniques we chose not to do these routine generalizations. On the other hand, to make sure that the interested reader has a full grasp of a generality of the results presented here, we do emphasize again that pretty much any distribution that can be pushed through the Lindeberg principle would work in place of the Gaussian one that we used.

\begin{singlespace}
\bibliographystyle{plain}
\bibliography{GardSphNegRefs}
\end{singlespace}

\end{document}